\title{\vspace{-1cm}Lagrangian aspects of Yudovich theory for 2D Euler}
\date{}
\documentclass[12pt]{article}
\usepackage{authblk}
\usepackage{blindtext}
\usepackage{amsfonts}
\usepackage{amsmath}
\usepackage{amsthm}
\usepackage{graphicx}
\usepackage{hyperref}
\usepackage{slashed}
\usepackage{comment}
\usepackage[margin=2cm]{geometry}
\usepackage{cancel}
\usepackage{collectbox}


\newtheorem{theorem}{Theorem}
\newtheorem{lemma}{Lemma}

\newtheorem{definition}{Definition}

\author{}
\begin{document}
\author{Theodore D. Drivas and Joonhyun La}
\maketitle
\vspace{-14mm}
\abstract{In this note, we establish Yudovich's existence  and uniqueness result for bounded (as well as mildly unbounded) vorticity weak solution of the two-dimensional incompressible Euler equations. As a biproduct of our proof, we establish some regularity results for the Yudovich solution map as it depends of the initial conditions and the fluid domain.}

\begin{center}
\textit{Dedicated to Prof. Peter Constantin, with admiration and respect.}
\end{center}
\vspace{-2mm}

Let $\Omega$ a bounded domain in the plane, or a two-dimensional manifold without boundary.  Provided the cohomology of $\Omega$ is trivial (see \cite{DE}), or else that $\Omega = \mathbb{T}^2$ or $\mathbb{T}\times[0,1]$, the two dimensional Euler equations in vorticity form for $\omega=\omega(t,x): \mathbb{R}\times \Omega\to \mathbb{R}$ read
\begin{align} \label{vorteqn1}
\partial \omega + u \cdot \nabla \omega &= 0 \qquad\qquad \text{in} \ \   \mathbb{R}\times \Omega\\
u &= K_\Omega[\omega] \qquad \! \text{in} \ \ \mathbb{R}\times \ \Omega\\
u\cdot n&= 0 \qquad\qquad \text{on} \ \  \mathbb{R}\times \partial \Omega\\
\omega &= \omega_0 \qquad \ \ \  \   \text{on} \ \  \{0\}\times  \Omega  \label{vorteqn4}
\end{align}
where $n$ is the normal to the boundary of $\Omega$ (if non-empty), and where  $K = \nabla^\perp \Delta^{-1}$ (with trivial Dirichlet boundary conditions for $\Delta$) is the Biot Savart operator on $\Omega$.  

Equations \eqref{vorteqn1}--\eqref{vorteqn4} show that the vorticity is simply transported within the domain by the velocity field it itself induces.  Namely, if the diffeomorphism $\Phi_t:\Omega \to \Omega$ solving
\begin{equation}	
\frac{d}{d t} X_t = u(X_t, t) \qquad X_0 = {\rm id}
\end{equation}
denotes the flowmap of the velocity field $u$, then   the solution to \eqref{vorteqn1}--\eqref{vorteqn4} ought to be 
\begin{equation}	
\omega(t) =  \omega_0 \circ X_t^{-1}.
\end{equation}
Of course, having such a representation requires that the velocity field be sufficiently regular to at least generate a flow of homeomorphisms, and this regularity must be determined by and consistent with the corresponding vorticity smoothness.  Classical results tell that of $\omega_0\in C^\alpha(\Omega)$ with $\alpha>0$, then there exists a unique solution in the class $u \in C^{1,\alpha}_{x,t}$, in particular the velocity is Lipschitz, which is more than sufficient to justify the above picture.  Indeed, in this classical setting the Lagrangian particle map is a flow of diffeomorphism of some finite regularity.  A major improvement to this result was realized first by Yudovich \cite{Yudovich1}, who proved that if $\omega_0\in L_x^\infty$, then there exists a unique \emph{weak solution} of class $\omega \in L^\infty_tL_x^\infty$.  Here, the velocity is no longer Lipschitz, but bounded vorticity is enough to ensure that the velocity is \emph{log-Lipschitz} and thus there exists a corresponding flow of H\"{o}lder continuous homeomorphisms with decaying regularity.  Subsequently, Yudovich improved this theory to allow slightly unbounded vorticies which essentially have initial $L^p$ norms allowed to grow as $p\to \infty$ just slightly slower than logarithmically \cite{Yudovich2}.  See \eqref{Yspace} for a precise definition.  The point here is that vorticity being in this space ensures that the velocity, while no longer necessarily log-Lipschitz, is still \emph{Osgood continuous} and thus by Osgood's  theorem generates a flow of homeomorphisms.   The Lagrangian picture persists to this setting, and coherent singular structures are know to exist \cite{DEL23}.

Although the Lagrangian picture is implied for Yudovich's theory, it is not how Yudovich constructed his solutions of proved uniqueness.  Here rather he relied on some involved quasi-Eulerian scheme and clever energy estimates. See also \cite{CS24}. To the best of our knowledge, we are not aware of a fully Lagrangian proof: for example, the formulation in \cite{MP} measured the $L^1$-difference of flows to prove uniqueness, yet it did not regard Euler equation as the ODE in the space of volume-preserving homeomorphisms. Also, we would like to highlight Lagrangian-Eulerian formulation of \cite{C}: here, all variables are Lagrangian, and the regularity is considered in $C^{1,\alpha}$ space for fluid velocity, where Eulerian formulation faces a regularity issue. Yet, as the formulation intends to and does provide a unified framework for well-posedness of a wide range of models which involve velocity gradient in their evolution equations, it comes back to Eulerian setting to formulate velocity and velocity gradient operators. On the other hand, in Yudovich theory, velocity gradient of velocity field is not needed. Our goal in this note is to provide a simple, self-contained and fully Lagrangian proof of Yudovich's result.  As a result of this proof, we can gain additional information about the structure of the theory.  In particular, we will prove that the Lagrangian solution map 
\begin{equation}
X_t = S_{t,\omega_0, \Omega}(X_0)
\end{equation}
is itself continuous at fixed times as a function of the initial data and domain, with regularity deteriorating at the same rate as that of the particle flowmap. Our main result is
\begin{theorem}
Let $\Omega$ be a $C^2$ bounded domain in the plane, or a surface without boundary, $\mathcal{M}$ be the space of measure-preserving homeomorphisms on $\Omega$  endowed with the $L^\infty$ topology and $L^1\cap\mathbf{Y}_{\Theta}$ be a Yudovich space  \eqref{Yspace}. Then the Euler solution map $S_{t,\omega_0, \Omega}:\mathcal{M}\rightarrow \mathcal{M} $ is well defined, $C^\infty$ smooth in time,  and  continuous  in its dependence on the initial data and domain (with a time-dependent modulus of continuity depending on $\Theta$).
\end{theorem}

For more precise statement about existence, uniqueness and dependence on initial data see Theorem \ref{theorem1}. For dependence on the domain, see Theorem \ref{theorem2}.   Smoothness in time is proved by Serfati \cite{Serfati} in the bounded vorticity case and generalized to the full Yudovich class by Sueur \cite{Sueur}. The particular modulus of continuity of the solution map depends on $\Theta$.  In that case, the  modulus for the solution map comes from inverting the function $\nu_\Theta$ given in  \eqref{numod}, e.g., $\nu_\Theta^{-1}( e^{ct} \nu_\Theta(\text{difference of parameters}))$.  In the special case of bounded vorticity, the modulus becomes 
a time-dependent H\"{o}lder modulus. Specifically
\begin{equation}
\|S_{t,\omega_0, \Omega}(X) - S_{t,\omega_0', \Omega}(X)\|_{L^\infty(\Omega)} \leq C  \|\omega_0-\omega_0'\|_{L^\infty(\Omega)}^{e^{-c t}}
\end{equation}
for appropriate constants $c,C>0$ depending only on $\|X\|_{L^\infty(\Omega)}$ and $\|\omega_0\|_{L^\infty(\Omega)}$, $\|\omega_0\|_{L^\infty(\Omega)}$.   In terms of the domain, we consider two related by an arbitrary twice-differentiable diffeomorphism $\Phi: \Omega\to \Omega' := \Phi(\Omega)$.  See \cite{L15,LZ19} for rougher domains. The solution map obeys:
\begin{equation}
\|S_{t,\omega_0, \Omega}(X) -  \Phi^{-1}\circ S_{t,\omega_0\circ \Phi^{-1}, \Phi(\Omega)}(X)\circ \Phi\|_{L^\infty(\Omega)} \leq C  \|\Phi - {\rm id}\|_{C^2(\Omega)}^{e^{-c t}},
\end{equation}
again for appropriate constants $c,C>0$. 
Note that these estimates correlates  precisely the well known deteriorating H\"older estimate for the regularity of the flowmap as a function of label.  For corresponding continuity properties of the \emph{Eulerian} solution map,  which are necessarily of a somewhat weaker nature, see the work of Constantin et al. \cite{CDE}.  See also  \cite{CW}. The reason for this estimate is that, viewed in Lagrangian coordinates, Euler becomes an ODE, parametrized by $\omega_0$ and $\Omega$, on the space $\mathcal{M}$ of measure-preserving homeomorphisms 
\begin{equation}
\dot{X}_t = V_{\omega_0,\Omega}[X_t]
\end{equation}
where, if $\omega_0\in L^1\cap\mathbf{Y}_{\Theta}$, then the ``vector field" $V_{\omega_0,\Omega}$ is Osgood continuous with Lipschitz dependence on the parameters.

Extensions to Navier-Stokes are possible due to the stochastic Lagrangian formulation of Constantin-Iyer \cite{CI}, see also \cite{CDE, CCS24}.

\section{Preliminaries}
We start with basic lemmas. Let $\Omega$ be $\mathbb{R}^2$, $\mathbb{T}^2$, or a bounded domain in $\mathbb{R}^2$ with a smooth boundary. Recall the following estimates on the Biot-Savart kernel
\begin{lemma}[\cite{MP}] \label{Kernelestim} Let $\Delta$ denote Dirichlet Laplacian, and $K = \nabla^\perp \Delta^{-1}$. Then
\begin{enumerate}
\item $| K  (x,y) | \le \frac{C}{|x-y|^{1} }$ for some constant $C$,
\item for $(x_1, y), (x_2, y)$ satisfying
\begin{equation} \label{yfar}
|x_1 - x_2 | \le \frac{1}{2} \max (|x_1 - y|, |x_2 - y|),
\end{equation}
then,  for some  constant $C$, we have
\begin{equation}
|K (x_1, y) - K  (x_2, y) | \le \frac{C |x_1 - x_2 | }{ (|x_1 - y| + |x_2 - y| )^{2}  }
\end{equation}
 Similarly, for $(x, y_1), (x, y_2)$ satisfying
\begin{equation} \label{xfar}
|y_1 - y_2| \le \frac{1}{2} \max (|x- y_1|, |x- y_2 | ),
\end{equation}
then,  for some  constant $C$, we have
\begin{equation}
|K (x, y_1) - K  (x, y_2) | \le \frac{C |y_1 - y_2 | }{(|x - y_1| + |x - y_2|)^{2} }.
\end{equation}
\end{enumerate}
\end{lemma}

We now introduce the notion of an Osgood modulus  of continuity

\begin{definition} For a number $a$, $0 < a \le 1$, a modulus of continuity is an increasing nonzero continuous function $\mu: [0, a] \rightarrow \mathbb{R}_{\ge 0}$ such that $\mu(0) = 0$. 
We call $\mu$ an Osgood modulus if
\begin{equation*}
\int_0 ^a \frac{dr}{\mu(r) } = \infty.
\end{equation*} 
\end{definition}

\begin{lemma}[Osgood lemma] \label{Osgood} 
Let $\rho: [t_0, T] \rightarrow [0, a]$ be a measurable function, $\gamma : [t_0, T] \rightarrow \mathbb{R}_{>0}$ be locally integrable, and $\mu: [0, a] \rightarrow \mathbb{R}_{>0}$ be an Osgood modulus of continuity. If for some $c \ge 0$ the following holds,
\begin{equation}
\rho(t) \le c (1+(t- t_0) ) + \int_{t_0} ^t \gamma (t') \mu (\rho(t') ) dt', t \in [t_0, T],
\end{equation}
then the following holds: if $c>0$, then for $t \in [t_0, T]$, 
\begin{equation}
-\mathcal{M} (\rho(t) ) + \mathcal{M} (c) \le \int_{t_0 } ^t \gamma(t') dt', \mathcal{M} (x) = \int_x ^a \frac{dr}{\mu(r) },
\end{equation}
and if $c=0$, then $\rho = 0$. 
\end{lemma}

\begin{lemma}\label{Picard} 
Suppose that $(\mathcal{X},d)$ is a complete metric space with some $\mathrm{Id} \in \mathcal{X}$, $\mu$ an Osgood modulus of continuity, and that $S: C([0, T]; \mathcal{X}) \rightarrow C([0, T];X)$ satisfies the following estimate: for every $t \in [0, T]$,
\begin{equation} \label{Quasi_Lipschitz}
\begin{split}
d(S[X ](t) , S[Y] (t) ) &\le d(S[X](0), S[Y](0) ) +   \int_0 ^t C (  \mu (d(X(t'), Y(t') )   ) + \mu (d(S[x](t'), S[Y] (t' ) ) )) dt'.
\end{split}
\end{equation} 
Then $S$ has a unique fixed point in $C([0, T]; \mathcal{X} ) \cap \{ X(t) : X(0) = \mathrm{id} \}$.
\end{lemma}
\begin{proof}
We employ the Picard scheme:
\begin{equation}
X^{k+1} = S[X^k], \qquad X^{0} =  \mathrm{id} 
\end{equation}
Then we have
\begin{equation}
d(X^{n+k} (t),  X^k(t) ) \le  C  \int_0 ^t \mu (d( X^{n+k} (t') , X^{k} (t') ) ) + \mu (d( X^{n+k-1} (t') , X^{k-1} (t') ) ) dt',
\end{equation}
and taking first supremum over $n$, then using non-decreasing property of $\mu$, and finally by defining
\begin{equation}
\rho^k (t) := \sup_n \sup_{0 \le t' \le t} d( X^{n+k} (t') , X^{k} (t') ),
\end{equation}
we have $\rho^k (t) \le  C \int_0 ^t \mu (\rho^k (t') ) + \mu (\rho^{k-1} (t') ) dt',$ and writing $\rho (t) = \limsup_{k\rightarrow \infty} \rho^k (t) $ we have $\rho(t) \le  C \int_0 ^t \mu (\rho (t')) dt'$. Here we have used Reverse Fatou's lemma with uniform upper bounds on $\mu(\rho^{k})$ using that $\mu$ is bounded. This implies that $\rho=0$: in other words, $\{X^k\}$ is a Cauchy sequence in $(M, d)$. Uniqueness immediately follows from \eqref{Quasi_Lipschitz} and Osgood lemma.
\end{proof}

Next we introduce the Yudovich spaces. Let $\mathfrak{C}_0 $ be a large constant ($\mathfrak{C}_0 = 1000$ works).
\begin{definition}
Let $\Theta : (\log \mathfrak{C}_0 , \infty) \rightarrow (0, \infty)$ be an increasing function satisfying 
\begin{equation} \label{Thetagrowth}
\int_{\log \mathfrak{C}_0 } ^\infty \frac{dp}{p \Theta(p) }  = \infty.
\end{equation}
We define the generalized Yudovich space associate with $\Theta$ by
\begin{equation}\label{Yspace}
\mathbf{Y}_{\Theta}  := \left \{ \omega_0 \in L^1 (\Omega) :\| \omega_0 \|_{\mathbf{Y}_{\Theta} }:= \sup_{p \in (\log \mathfrak{C}_0 , \infty) } \frac{ \| \omega_0 \|_{L^p} }{\Theta(p) } < \infty \right \}.
\end{equation}
For each $\Theta$, we define an associated modulus of continuity
\begin{equation}\label{Mutheta}
\mu_{\Theta} (r) := \begin{cases}
r \log (1/r) \Theta (\log (1/r) ), 0 < r \le \mathfrak{C}_0^{-1}, \\
\mathfrak{C}_0^{-1} \log \mathfrak{C}_0  \Theta (\log \mathfrak{C}_0 ) , r \ge \mathfrak{C}_0^{-1}.
\end{cases}
\end{equation}
We define the associated function $\mathcal{M}_{\Theta}$ by
\begin{equation}\label{Mmod}
\mathcal{M}_{\Theta} (r) = \int_r ^{\mu_{\Theta} (\mathfrak{C}_0^{-1} ) } \frac{ds}{\mu_{\Theta} (s) }.
\end{equation}
\end{definition}

It is easy to check that $\mu_{\Theta}$ is an Osgood modulus of continuity, thanks to \eqref{Thetagrowth}. Finally, we introduce a new modulus of continuity
\begin{equation}\label{numod}
\nu_{\Theta} (r) = \exp ( - \mathcal{M}_{\Theta} (r) ).
\end{equation} 
We present the following useful lemma.
\begin{lemma} \label{pointwise}
Let $Z$ be a measure-preserving homeomorphism on $\Omega$, and let $M >0$, $\ell  \in [1, \mathfrak{C}_0 ]$. Also let $\omega_0 \in L^1 \cap \mathbf{Y}_{\Theta}$. Then for some $q>1$,
\begin{equation} \label{pointwiseestimate}
\begin{split}
\left | \int_{\Omega } \omega_0 (b)  \mathbf{1}_{|Z(b) | \ge 1} (b) \frac{1}{|Z(b) | }  db \right | &\le C \| \omega_0 \|_{L^1}, \\
 \left | \int_{\Omega } \omega_0 (b) M \mathbf{1}_{|Z(b) | \ge 1} (b) \frac{1}{|Z(b) |^2 }  db \right |  &\le  C_q M \| \omega_0 \|_{L^q}, \\
\left | \int_{\Omega } \omega_0 (b) \mathbf{1}_{1 \ge |Z(b) | \ge M } (b) \frac{M}{|Z(b)|^2}   db \right | &\le C \mu_{\Theta} (M\mathfrak{C}_0^{-1} ) \| \omega_0 \|_{\mathbf{Y}_\Theta}, \\
\left | \int_{\Omega} \omega_0 (b) \mathbf{1}_{\ell M \ge |Z(b) | } (b) \frac{1}{|Z(b) | }   db \right | &\le \begin{cases} C  \mu_{\Theta} (M\mathfrak{C}_0^{-1} ) \| \omega_0 \|_{\mathbf{Y}_{\Theta} }, \qquad M \in (0, 1), \\ C_q  \|\omega_0\|_{L^{1} \cap L^{2q}}  \qquad M \ge 1.
\end{cases}
\end{split}
\end{equation}
\end{lemma}
\begin{proof}
The first two inequalities are immediate, and also the case $M\ge 1$: in fact the third inequality is vacuous when $M \ge 1$. We only need to check the last two in the case $M\in (0,1)$. For the third inequality, for each $p > \log \mathfrak{C}_0 $ we have
\begin{equation}
\begin{split}
&\left | \int_{\Omega} \omega_0 (b) \mathbf{1}_{1 \ge |Z(b) | \ge M } (b) \frac{M}{|Z(b)|^2}   db \right | \le C \| \omega_0 \|_{L^p}   M \int_M ^1 r^{1 - 2p'} dr \\
&\le C  \| \omega_0 \|_{\mathbf{Y}_{\Theta} }  \Theta(p) \frac{1}{2(1-p')} M (1-M^{2(1-p')} ) \le C \| \omega_0 \|_{\mathbf{Y}_{\Theta} } \Theta(p) (p-1) M^{1-\frac{2}{p-1} }.
\end{split}
\end{equation}
Since the constant $C$ does not depend on $p$, we conclude that 
\begin{equation}
\begin{split}
\left | \int_{\Omega} \omega_0 (b) \mathbf{1}_{1 \ge |Z(b) | \ge M } (b) \frac{M}{|Z(b)|^2}   db \right | \le  C \inf_{p > \log \mathfrak{C}_0 }  \| \omega_0 \|_{\mathbf{Y}_{\Theta} } \Theta(p) (p-1) M^{1-\frac{2}{p-1} }.
\end{split}
\end{equation}
Similarly, for $M \in (0,1)$, 
\begin{equation}
\begin{split}
&\left | \int_{\Omega} \omega_0 (b) \mathbf{1}_{\ell M \ge |Z(b) | } (b) \frac{1}{|Z(b) | }   db \right | \le C \| \omega_0 \|_{L^p} \int_0 ^{\ell M} r^{1-p'} dr \\
& = C \| \omega_0 \|_{\mathbf{Y}_{\Theta} } \Theta(p) \frac{1}{2-p'}  \ell^{2-p'} M^{2-p'} \le C \| \omega_0 \|_{\mathbf{Y}_{\Theta} } \Theta(p) M^{1- \frac{1}{p-1} },
\end{split}
\end{equation}
using $1/(2-p') < C(\mathfrak{C}_0 ) $ for $p>\log \mathfrak{C}_0 $ and that $\ell \in [1, \mathfrak{C}_0 ]$. Again we conclude that
\begin{equation}
\begin{split}
&\left | \int_{\Omega} \omega_0 (b) \mathbf{1}_{\ell M \ge |Z(b) | } (b) \frac{1}{|Z(b) | }   db \right | \le
C \| \omega_0 \|_{\mathbf{Y}_{\Theta} } \inf_{p > \log \mathfrak{C}_0 } \Theta(p) M^{1-\frac{1}{p-1} }.
\end{split}
\end{equation}
We obtain the conclusion by taking $p = \log \frac{\mathfrak{C}_0}{ M } (>\log \mathfrak{C}_0 )$.
\end{proof}

\section{Yudovich solution operator}

\begin{definition}
We denote $\mathcal{M}$ by the space of measure-preserving homeomorphisms on $\Omega$, and equip the distance $d(X, Y) = \| X - Y \|_{L^\infty}$.
\end{definition}

\begin{definition}
For $T>0$ and $\omega_0 \in L^1 \cap \mathbf{Y}_{\Theta} $, we define $S: C([0, T]; \mathcal{M}) \rightarrow C([0, T] ; \mathcal{M} ) $
\begin{equation} \label{solnop}
\begin{split}
\tilde{X} &= S[X] = S_{\omega_0} [X], \qquad \frac{d}{dt} \tilde{X} (a,t) = v_{X(t)} (\tilde{X}(a,t)), \qquad \tilde{X} (\cdot, 0) =  \mathrm{id}  \\
v_X (y) &= \int_{\Omega} K(y, X(b) ) \omega_0 (b) db.
\end{split}
\end{equation}
\end{definition}
In order to make sense of the definition, we need to prove three: first, the Cauchy problem in \eqref{solnop} admits the unique solution, second, $\tilde{X}(t) \in \mathcal{M}$, and last, $\tilde{X}(t)$ is continuous on $t$. The first claim is immediate: we first note that if $\phi \in L^1$ and $X \in \mathcal{M}$, then $\phi \circ X, \phi \circ X^{-1} \in L^1$ and  
\begin{equation*}
\int_{\Omega} \phi (x) dx = \int_{\Omega} \phi \circ X(a) da = \int_{\Omega} \phi \circ X^{-1} (x) dx.
\end{equation*}
Then we observe that if $\omega_0 \in L^1 \cap \mathbf{Y}_{\Theta}$, then so is $\omega_0 \circ X^{-1}$ with its norm unchanged. Finally, we see that for each $y \in \Omega $, $x \rightarrow K(y, x) (\omega_0 \circ X^{-1}) (x) $ is $L^1$: we have 
\begin{equation*}
\begin{split}
\int_{\Omega} \left | K(y,x) (\omega_0 \circ X^{-1}) (x) \right |dx &\le \int_{|y-x| \ge 1 }   | \omega_0 \circ X^{-1} (x) | dx  + \int_{|y-x| <1}  \frac{1}{|y-x|} | \omega_0 \circ X^{-1} (x) | dx \\
&\le \| \omega_0 \|_{L^1} + C_p \| \omega_0  \|_{L^p}
\end{split}
\end{equation*}
for some $p>\log \mathfrak{C}_0$. Therefore, we see that $v_X (y) = \nabla^\perp \Delta^{-1} (\omega_0 \circ X^{-1} ) (y)$ and it is well-known that for $\omega_0 \circ X^{-1} \in L^1 \cap \mathbf{Y}_{\Theta}$, $\nabla^\perp \Delta^{-1} (\omega_0 \circ X^{-1} ) $ has an Osgood modulus of continuity $\varphi_{\Theta}$, depending only on $\Theta$, and thus admitting unique flow. 

Next, we prove the last claim. It suffices to show the continuity of $S$ on time at $t=0$: other cases can be shown in the same manner. 
\begin{lemma} \label{timeconti}
We have the following estimate: for any $q > 2$,
\begin{equation}
d(\tilde{X}(t), \mathrm{Id } ) \le C_q \| \omega_0 \|_{L^1 \cap L^q} t + C \|\omega_0 \|_{L^1 \cap \mathbf{Y}_{\Theta} } \int_0 ^t  \mu_{\Theta}  ( d (\tilde{X}(s), \mathrm{id} ) ) ds.
\end{equation}
\end{lemma}
\begin{proof}
Suppose that $|\tilde{X}(a,t) - a| \le \mathfrak{C}_0^{-1} $. We have
\begin{equation}
\begin{split}
\frac{d}{dt} \left ( \tilde{X}(a,t) - a \right ) &=  \int_{\Omega} \left ( K(\tilde{X}(a,t), X(b,t) ) - K(a, X(b, t)) \right ) \omega_0 (b ) db  \\
&+ \int_{\Omega} K(a, X(b,t) ) \omega_0 (b) db.
\end{split}
\end{equation}
First we see that $\left | \frac{1}{2\pi} \int_{\Omega} K(a, X(b,t)) \omega_0 (b) db \right | \le C_q \| \omega_0 \|_{L^1 \cap L^q}$ for any $q>2$. Next, we let
\begin{equation}
F_{a} (b) = K(\tilde{X}(a,t), X(b,t) ) - K(a, X(b,t) )
\end{equation}
and 
\begin{equation}
M = \mathfrak{C}_0 |\tilde{X}(a,t) - a|.
\end{equation}
Now if $|X(b,t) - \tilde{X}(a,t) | \ge M$, then $|X(b,t) - \tilde{X}(a,t) | \ge 2 |\tilde{X}(a,t) - a|$ (since $\mathfrak{C}_0$ is large) and
\begin{equation*}
\begin{split}
|a - X(b,t) | &\ge |X(b, t) - \tilde{X}(a,t) | - \mathfrak{C}_0^{-1} M \ge (1-\mathfrak{C}_0^{-1}) M \ge 2 |\tilde{X}(a,t) - a|, \\
|a - X(b,t) | &\ge |X(b, t) - \tilde{X}(a,t) | - \mathfrak{C}_0^{-1} M \ge (1- \mathfrak{C}_0^{-1}) |X(b,t) - \tilde{X}(a,t) |.
\end{split}
\end{equation*}
On the other hand, if $|X(b,t) - \tilde{X}(a,t) | \le M$, then 
\begin{equation}
|a - X(b,t) | \le |X(b, t) - \tilde{X}(a,t ) | + |\tilde{X}(a,t) - a | \le (1+\mathfrak{C}_0^{-1}) M.
\end{equation}
Therefore, we have
\begin{equation*}
\begin{split}
|F_a (b) | &\le   \frac{C\mathbf{1}_{ |X(b,t) - \tilde{X}(a,t) | \ge 1 } (b) }{| X(b,t) - \tilde{X} (a,t) |} +  \frac{CM \mathbf{1}_{1 \ge |X(b,t) - \tilde{X}(a,t) | \ge M } (b) }{|X(b,t) - \tilde{X}(a,t) |^2 } +  \frac{C \mathbf{1}_{M \ge |X(b,t) - \tilde{X}(a,t) | } (b) }{|X(b,t) - \tilde{X}(a,t) | } \\
&\qquad + \frac{C\mathbf{1}_{(1+\mathfrak{C}_0^{-1}) M \ge |X(b,t) - a | } (b) }{|X(b,t) - a | },
\end{split}
\end{equation*}
if $M \le 1$, and if $M>1$, we have $\left |\int_{\Omega} F_a (b) \omega_0 (b) db \right | \le C_q \| \omega_0 \| _{L^1 \cap L^q}$ for any $q>2$.
Thus, by Lemma \ref{pointwise}
\begin{equation}
\frac{d}{dt} |\tilde{X}(a,t) - a|  \le C_q \| \omega_0 \|_{L^1 \cap L^q} + C \| \omega_0 \| _{L^1 \cap \mathbf{Y}_{\Theta} } \mu_{\Theta} (|\tilde{X}(a,t) - a|)
\end{equation} 
for any $q > 2$. Integrating over time gives the desired conclusion.
\end{proof}

Next, we show that $\tilde{X}(t) \in \mathcal{M}$. In fact, we show that $\tilde{X}(t)$ is measure-preserving and has a definite modulus of continuity (which may deteriorate over time). We first show that $\tilde{X}$ is measure-preserving.
\begin{lemma} \label{volpres}
For each $t \in [0, T]$, $\tilde{X}(t)$ is continuous. Also, for a continuous function $\phi$ with compact support, 
\begin{equation}
\int_{\Omega} \phi(\tilde{X} (a,t) ) da = \int_{\Omega} \phi(y) dy.
\end{equation}
\end{lemma}
\begin{proof}
We note that if $\omega_0 \in H^s \cap L^1 \cap \mathbf{Y}_{\Theta}$ for sufficiently large $s>1$, then the statement is true: we can define $\nabla_a \tilde{X}$, and $\mathrm{det} \nabla_a \tilde{X}$ can easily be shown to be 1. Now we denote $\omega_0 ^\varepsilon = \omega_0 \star \phi_\varepsilon$, where $\phi_\varepsilon (x) = \frac{1}{\varepsilon} \phi (\varepsilon^{-1} x )$ with $\phi$ being a standard mollifier, and let $\tilde{X}_\varepsilon = S_{\omega_0^\varepsilon} [X].$ Since $\sup_{\varepsilon>0} \| \omega_0 ^\varepsilon \|_{L^1 \cap \mathbf{Y}_{\Theta} } \le \| \omega_0  \|_{L^1 \cap \mathbf{Y}_{\Theta} }< \infty$, we see that associated velocity fields $v_X ^{(\varepsilon)}$ for $S_{\omega_0 ^\varepsilon}$ have uniform Osgood modulus $\varphi_{\Theta}$. 

Next we show that $v_X ^{(\varepsilon)}$ converges to $v_X$ in $L^\infty$: note that for every $y \in \mathbb{R}^2$,
\begin{equation*}
v_X ^{(\varepsilon)} (y) - v_X (y) = \int_{\Omega} K(y,x) (\omega_0 - \omega_0 ^\varepsilon ) \circ X^{-1} (x) dx,
\end{equation*}
and therefore for any $q>2$, $\| v_X^{(\varepsilon)} - v_X \|_{L^\infty} \le  C_q \| \omega_0 - \omega_0 ^\varepsilon \|_{L^1 \cap L^q} \rightarrow 0,$ as $\varepsilon \rightarrow 0$.

Finally, we have
\begin{equation}
\frac{d}{dt} (\tilde{X} - \tilde{X}_{\varepsilon} ) (a,t) = \left (v_X (\tilde{X} (a,t)) - v_X (\tilde{X}_{\varepsilon} (a,t) ) \right ) + (v_X - v_X^{(\varepsilon)} ) (\tilde{X}_{\varepsilon} (a,t) )
\end{equation}
and thus
\begin{equation}
\frac{d}{dt} \| \tilde{X} - \tilde{X}_{\varepsilon } \|_{L^\infty} \le C \varphi_{\Theta} (\| \tilde{X} - \tilde{X}_{\varepsilon } \|_{L^\infty} ) + \| v_X^{(\varepsilon)} - v_X \|_{L^\infty} ,
\end{equation}
and Osgood's lemma gives us the uniform convergence of $\tilde{X}_{\varepsilon}$ to $\tilde{X}$. This shows that $\tilde{X}$ is continuous. Now, since from Lemma \ref{timeconti} and Osgood's lemma we see that $\sup_{t \in [0, T]} d (\tilde{X} (t), \mathrm{id} ) \le C(T, \| \omega_0 \|_{L^1 \cap \mathbf{Y}_{\Theta} } )$, which implies that there exists a constant $C := C(T, \| \omega_0 \|_{L^1 \cap \mathbf{Y}_{\Theta} } )$ such that for every $\varepsilon>0$
\begin{equation}
\mathrm{supp} \phi \circ \tilde{X}_{\varepsilon}, \mathrm{supp} \phi \circ \tilde{X} \subset \mathrm{supp} \phi + B_{C} (0).
\end{equation}
Then we obtain the desired result from interchangeability of uniform convergence and integration.
\end{proof}
Lemma \ref{volpres} implies, by density argument, that $\tilde{X}(t)$ is measure-preserving. Thus $\tilde{X}(t) \in \mathcal{M}$. Therefore, $S : C([0, T] ; \mathcal{M} ) \rightarrow C([0, T] ; \mathcal{M} )$ is well-defined. 

Next, we show that $\tilde{X}$ enjoys a modulus of continuity, which is uniform over $X \in \mathcal{M}$.
\begin{lemma} \label{Xregular}
We have the following estimate:
\begin{equation}
-\mathcal{M}_{\Theta} (|\tilde{X}(a,t) - \tilde{X} (b,t) | ) \le - \mathcal{M}_{\Theta} (|a-b| ) + C  \| \omega_0 \|_{L^1 \cap \mathbf{Y}_{\Theta} }t.
\end{equation}
\end{lemma}
\begin{proof}
Let $a, b \in \Omega$, and $M = \mathfrak{C}_0 |\tilde{X} (a,t) - \tilde{X} (b,t) |$. 
\begin{equation}
\begin{split}
\frac{d}{dt} (\tilde{X}(a,t) - \tilde{X} (b,t) ) &=  \int_{\Omega } \left ( K (\tilde{X}(a,t), X(c,t) ) - K(\tilde{X}(b,t), X(c,t) ) \right ) \omega_0 (c) dc.
\end{split}
\end{equation}
Now if $|X(c,t) - \tilde{X}(a,t) | \ge M$, 
\begin{equation}
\begin{split}
|\tilde{X}(b,t) - X(c,t) | &\ge |X(c,t) - \tilde{X}(a,t) | - |\tilde{X}(a,t) - \tilde{X}(b,t) | \ge (1-\mathfrak{C}_0^{-1}) |X(c,t) - \tilde{X}(a,t) |, \\
|\tilde{X}(a,t) - X(c,t) | &\ge 2 | \tilde{X}(a,t) - \tilde{X}(b,t) |,  \\
|\tilde{X}(b,t) - X(c,t) | &\ge |X(c,t) - \tilde{X}(a,t) | - |\tilde{X}(a,t) - \tilde{X}(b,t) | \ge 2 |\tilde{X} (a,t) - \tilde{X}(b,t) |,
\end{split}
\end{equation}
and if $|X(c,t) - \tilde{X}(a,t) | \le M$,
\begin{equation}
|\tilde{X}(b,t) - X(c,t) | \le |X(c,t) - \tilde{X}(a,t) | + |\tilde{X}(a,t) - \tilde{X}(b,t) | \le (1+\mathfrak{C}_0^{-1}) |\tilde{X}(a,t) - \tilde{X}(b,t) |.
\end{equation}
To summarize, if $0 < M < 1$, we have
\begin{equation}
\begin{split}
&\left | K (\tilde{X}(a,t), X(c,t) ) - K(\tilde{X}(b,t), X(c,t) ) \right | \le  \frac{C M  \mathbf{1}_{|X(c,t) - \tilde{X}(a,t) | \ge 1 } (c)}{|\tilde{X}(a,t) - X(c,t) |^2} \\
&+  \frac{CM  \mathbf{1}_{1 \ge |X(c,t) - \tilde{X}(a,t) | \ge M } (c)}{|\tilde{X}(a,t) - X(c,t) |^2} + \frac{ C \mathbf{1}_{|X(c,t) - \tilde{X}(a,t) | \le M}(c) }{|\tilde{X}(a,t) - X(c,t) | } +  \frac{C \mathbf{1}_{|X(c,t) - \tilde{X}(b,t) | \le  (1+\mathfrak{C}_0^{-1})M}(c)}{|\tilde{X}(b,t) - X(c,t) | } 
\end{split}
\end{equation}
for $M \in (0,1)$. For $M \ge 1$, the first term $C \mathbf{1}_{|X(c,t) - \tilde{X}(a,t) | \ge 1 } (c) \frac{M}{|\tilde{X}(a,t) - X(c,t) |^2} $ is replaced by $C_q \| \omega_0 \|_{L^1 \cap L^q} $ for some $q > 2$. Then Lemma \ref{pointwise} gives
\begin{equation}
\frac{d}{dt} | \tilde{X}(a,t) - \tilde{X}(b,t) | \le C \| \omega_0 \|_{L^1 \cap \mathbf{Y}_{\Theta} }  \mu_{\Theta} (|\tilde{X}(a,t) - \tilde{X}(b,t) |)
\end{equation}
and integration over time and application of Osgood's lemma gives us the desired conclusion.
\end{proof}
Finally, we prove that \eqref{Quasi_Lipschitz} holds.  
\begin{lemma}
Let $\tilde{X} = S_{\omega_0} [X]$ and $\tilde{Y} = S_{\omega_1} [Y]$. Then for every $p > 2$,
\begin{equation} 
\frac{d}{dt} d(\tilde X(t), \tilde Y(t)) \le C \| \omega_0 - \omega_1 \|_{L^1 \cap L^p} + C \|\omega_0 \|_{L^1 \cap \mathbf{Y}_{\Theta} } (\mu_{\Theta} ( d(X(t), Y(t)) ) + \mu_{\Theta} (d(\tilde{X}(t), \tilde{Y} (t) ) ) ).
\end{equation}
\end{lemma}

\begin{proof}
We write $\tilde{X} = S_{\omega_0} [X], \ \tilde{Y} = S_{\omega_1} [Y]$ and 
\begin{equation}
M:= \mathfrak{C}_0 \max (\| X- Y \|_{L^\infty}, \| \tilde X - \tilde Y \|_{L^\infty } ). 
\end{equation}
For any $a \in \Omega$, 
\begin{equation}
\begin{split}
\frac{d}{dt} &(\tilde{X}(a,t) - \tilde Y(a,t) ) = \int_{\Omega} K(\tilde Y(a,t), Y(b,t) ) (\omega_0 (b) - \omega_1 (b) ) db \\
&+ \int_{|X(b,t) - \tilde X(a,t)| \ge M}  \big ( K(\tilde X(a,t), X(b,t))  - K(\tilde Y(a,t), X(b,t) ) \big ) \omega_0 (b) db \\
&+ \int_{|X(b,t) - \tilde X(a,t)| \ge M}  \big ( K(\tilde Y(a,t), X(b,t))  - K(\tilde Y(a,t), Y(b,t) ) \big ) \omega_0 (b) db \\
&+ \int_{|X(b,t) - \tilde X(a,t)| \le M } K(\tilde X(a,t), X(b,t)) \omega_0 (b) db-  \int_{|X(b,t) - \tilde \tilde X(a,t)| \le M} K(\tilde Y(a,t), Y(b,t) ) \omega_0 (b) db.
\end{split}
\end{equation}
Suppose first that $M<1$. If
\begin{equation} \label{abfar}
|X(b,t)-\tilde X(a,t)| \ge M,
\end{equation} 
then
\begin{equation*}
\begin{split}
|\tilde Y(a, t) - Y(b,t) | &\ge |\tilde X(a,t) - X(b,t) | - |\tilde X(a,t) - \tilde Y(a, t) | - |X(b,t) - Y(b,t) | \\
&\ge (1- 2 \mathfrak{C}_0^{-1} ) |\tilde X(a,t) - X(b,t) |  \ge M/2,  \\
|\tilde Y(a, t) - Y(b,t) | &\ge |\tilde X(a,t) - X(b,t) | - \| \tilde X - \tilde Y \|_{L^\infty } - \| X - Y \|_{L^\infty}  \ge \frac{1}{2} |X(b,t) - \tilde X(a,t) | \ge M/2 , \\
|\tilde Y(a,t) - X(b,t) | &\ge |\tilde X(a,t) - X(b,t) | - |\tilde X(a,t) - \tilde Y(a,t) | \ge \frac{1}{2} |X(b,t) - \tilde X(a,t) |  \ge M/2 \\
|\tilde Y(a,t) - X(b,t) | &\ge | \tilde Y(a,t) - Y (b,t) | - |X(b,t) -  Y(b,t) | \ge \frac{1}{2} |Y(b,t) - \tilde Y (a,t) |, 
\end{split}
\end{equation*}
so by the first two inequalities, \eqref{yfar} is satisfied with $x_1 = X(a,t)$, $x_2 = Y(a,t)$, $y = X(b,t)$, and also \eqref{xfar} is satisfied with $x = Y(a,t)$, $y_1 = X(b,t)$, and $y_2 = Y(b,t)$. Thus, for $a, b$ satisfying \eqref{abfar} we have
\begin{equation*}
\begin{split}
\left | K(\tilde X(a,t), X(b,t) ) - K(\tilde Y(a,t), X(b,t) ) \right | &\le  C \frac{|\tilde X(a,t) - \tilde Y(a,t) | }{|\tilde X(a,t) -X(b,t) |^2}  \le C \frac{ M } {|\tilde X(a,t) - X(b,t) |^2} , \\
\left | K(\tilde Y(a,t), X(b,t) ) - K(\tilde Y(a,t), Y(b,t) ) \right | &\le C \frac{|X(b,t) - Y(b,t) | }{ |\tilde Y(a,t) - Y(b,t) | ^2 } \le C \frac{M }{|\tilde{X}(a,t) - X(b,t) |^2}.
\end{split}
\end{equation*}
On the other hand, if
\begin{equation}\label{abclose}
|X(b,t) -\tilde X(a,t) | \le M,  
\end{equation}
then we have
\begin{equation}
\begin{split}\nonumber
|\tilde Y(a, t) - Y(b,t) | &\le |\tilde X(a,t) -  X(b,t) | +|\tilde X (a,t) - \tilde Y(a, t) | + | X(b,t) - Y(b,t) | \le (1+2 \mathfrak{C}_0 ^{-1} ) M.
\end{split}
\end{equation}
Let 
\begin{equation}
F_{a} (b) = K(\tilde{X} (a,t), X(b,t) ) - K(\tilde{Y} (a,t) , Y(b,t) ).
\end{equation}
We have
\begin{equation}
\begin{split}
|F_{a} (b) | &\le   \frac{CM \mathbf{1}_{|X(b,t) - \tilde{X}(a,t) | \ge 1 } (b)}{|\tilde{X} (a,t) - X(b,t) |^2}  + \frac{C M  \mathbf{1}_{1 \ge |X(b,t) - \tilde{X}(a,t) | \ge M } (b) }{|\tilde{X} (a,t) - X(b,t) |^2} \\
&\qquad +  \frac{C \mathbf{1}_{|X(b,t) - \tilde{X}(a,t) | \le M } (b)}{|\tilde{X}(a,t) - X(b,t) | } + \frac{ C \mathbf{1}_{|\tilde{Y}(a,t) - Y(b,t) | \le (1+2 \mathfrak{C}_0^{-1} ) M } (b)}{|\tilde{Y}(a,t) - Y(b,t) | }, \\
\end{split}
\end{equation}
and if $M \ge 1$, the first term is replaced by
\begin{equation}
 \frac{C \mathbf{1}_{|X(b,t) - \tilde{X} (a,t) | \ge 1} (b)}{|\tilde{X}(a,t) - X(b,t) |} + \frac{C}{|\tilde{Y}(a,t) - Y(b,t) |},
\end{equation}
and by Lemma \ref{pointwise} we obtain the desired estimate.
\end{proof}

Putting $\omega_0 = \omega_1$, we apply Lemma \ref{Picard} for each $\omega_0 \in L^1 \cap \mathbf{Y}_{\Theta} $ we have unique fixed point $X = S_{\omega_0} [X]$. Now let $X = S_{\omega_0} [X], Y = S_{\omega_1} [Y]$. Then by Osgood lemma, we have the following dependence on initial data.
\begin{theorem} \label{theorem1}
For every $\omega_0 \in L^1 \cap \mathbf{Y}_{\Theta} $  there exists a unique weak solution of Euler in $\omega(t) \in L^1 \cap \mathbf{Y}_{\Theta} $. Moreover, for every $p > 2$, there exists a universal constant $C>0$ such that
\begin{equation}
\nu_{\Theta} (d(X(t), Y(t) ) ) \le \exp (Ct \| \omega_0 \|_{L^1 \cap \mathbf{Y}_{\Theta} } ) \nu_{\Theta} (C\| \omega_0 - \omega_1 \|_{L^1 \cap L^p }  ).
\end{equation}
\end{theorem}

\subsection{Comparison between two domains}
Let $\Omega_1, \Omega_2$ be two bounded, simply connected domains in $\mathbb{R}^2$ with smooth boundaries, and let $\Phi: \Omega_1 \rightarrow \Omega_2$ be a measure-preserving diffeomorphism between two. Also we denote 
$K_1 = \nabla^\perp \Delta_1^{-1}, \qquad K_2 = \nabla^\perp \Delta_2 ^{-1}$, where $\Delta_i, i=1, 2 $ are Dirichlet Laplacian on $\Omega_i$. Also, we denote $G_i^{-1}$ be Green's function for Dirichlet Laplacian on $\Omega_i$. We have
\begin{equation*}
\frac{d}{dt} X_i (a, t) = \int_{\Omega_i} K_i (X_i (a, t), X_i (b,t) ) \omega_0 ^i (b) db, i=1, 2.
\end{equation*}
Now we consider $\tilde{X}_1 = \Phi^{-1} \circ X_2 \circ \Phi$: assuming $\omega_0 ^1 = \omega_0^2 \circ \Phi = \omega_0$, we have
\begin{equation}
\begin{split}
&\frac{d}{dt} (\tilde{X}_1  - X_1) (a,t) \\
&= \int_{\Omega_1} \left (K_2 ( X_2 ( \Phi(a), t) , X_2 (\Phi(b), t) ) \cdot \nabla \Phi^{-1} (X_2 (\Phi (a), t) ) - K_1 (X_1 (a,t) , X_1 (b,t) ) \right ) \omega_0 (b) db \\
&= \int_{\Omega_1} \left [ K_2 ( \Phi \circ \tilde{X}_1 (a), \Phi \circ \tilde{X}_1 (b) ) \cdot (\nabla \Phi^{-1}) \circ \Phi (\tilde{X}_1 (a) ) - K_1 (X_1(a, t) , X_1 (b,t) )  \right ] \omega_0 (b) db.
\end{split}
\end{equation}
Next, remark that $K_1(x,y) = \nabla_x^{\perp} (G_1 (x,y) )$ and 
\begin{equation*}
K_2 (\Phi(x), \Phi(y) ) \cdot (\nabla \Phi^{-1} ) \circ \Phi (x) = \nabla_x^\perp (G_2 (\Phi(x), \Phi(y) ) ).
\end{equation*}
Therefore, we have
\begin{equation}
\begin{split}
\frac{d}{dt} (X_1 - \tilde{X}_1 )(a,t) &= \int_{\Omega_1} [ K_1 (X_1 (a,t), y) - K_1 (\tilde{X}_1 (a,t) , y)  ] \omega_0 \circ X_1^{-1} (y) dy\\
&+ \int_{\Omega_1} \left (K_1 (\tilde{X}_1 (a,t), X_1 (b,t) ) - K_1 (\tilde{X}_1 (a,t), \tilde{X}_1 (b,t) ) \right )  \omega_0 (b) db \\
&+ \left. \nabla_x^{\perp} \left [ \int_{\Omega_1} \left (G_1 (x,y)  - G_2 (\Phi(x), \Phi(y) ) \right ) \omega_0 \circ \tilde{X}_1^{-1} (y)    dy \right ] \right |_{x =\tilde{X}_1 (a,t) }.
\end{split}
\end{equation}
The first and second terms are bounded by, as before, $C \| \omega_0 \|_{L^1 \cap \mathbf{Y}_{\Theta} } \mu_{\Theta} (| (X_1 - \tilde{X}_1 (a,t) | )$. To investigate the last term, we first note that if we write
\begin{equation}
F(x) = \int_{\Omega_1} G_2 (\Phi (x), \Phi (x') ) G(x') dx', 
\end{equation}
then $f = F \circ \Phi^{-1}, g = G \circ \Phi^{-1}$ solves 
\begin{equation*}
\Delta f(y) = g(y), y \in \Omega_2, \qquad f(y) = 0, y \in \partial \Omega_2.
\end{equation*}
By the chain rule, $\partial_{y_i} (F\circ \Phi^{-1} ) = (\partial_{y_i} \Phi^{-1} _j ) (\partial_{x_j} F) \circ \Phi^{-1} = \left [\partial_{x_j}^\perp \Phi_i^{\perp} (\partial_{x_j } F) \right ] \circ \Phi^{-1},$ and we have the following elliptic equation for $(F, G)$:
\begin{equation}
\left \{
\begin{split}
LF (x) := &\partial_{k} ( \partial_k^\perp \Phi_i \partial_j^\perp \Phi_i \partial_j F) (x) = G (x), \qquad x \in \Omega_1, \\
F(x) &= 0, x \in \partial \Omega_1.
\end{split} \right.
\end{equation}
Now we denote $\psi (x) = \int_{\Omega_1} G_1 (x,y) \omega_0 \circ \tilde{X}_1 ^{-1} (y) dy$, $\tilde{\psi} (x) = \int_{\Omega_1} G_2 (\Phi (x), \Phi(y) ) \omega_0 \circ \tilde{X}_1 ^{-1} (y) dy$. Then the last term is $\left.\nabla_x^\perp (\psi - \tilde{\psi} )(x)  \right |_{x = \tilde{X}_1 (a,t) }$, and $\psi - \tilde{\psi}$ solves the following:
\begin{equation}
\left \{
\begin{split}
\Delta (\psi - \tilde{\psi} )(x)  &= (L - \Delta) \tilde{\psi} = \partial_k \left [ (\partial_k^\perp \Phi_i \partial_j^\perp \Phi_i ) - \delta_{jk} \right ] \partial_j \tilde{\psi} (x),  \qquad x \in \Omega_1 \\
(\psi - \tilde{\psi}) (x) &= 0, x \in \partial \Omega_1.
\end{split} \right.
\end{equation}
Therefore, the last term is $\nabla^\perp \Delta_1 ^{-1} (L- \Delta) \tilde{\psi} = \nabla^\perp \Delta_1 ^{-1} (L- \Delta) L^{-1} \omega_0 \circ \tilde{X}_1 ^{-1} $, and we have, for any $q>2$, 
\begin{equation*}
\| \nabla \Delta_1 ^{-1} (L - \Delta) L^{-1} \omega_0 \circ \tilde{X}_1 ^{-1}  \|_{L^\infty} \le C_q \| (L - \Delta) L^{-1} \omega_0 \circ \tilde{X}_1 ^{-1} \|_{L^q} \le C_q  \| \nabla \Phi - \mathbb{I} \|_{C^1} C(\| \nabla \Phi \|_{C^1} ) \| \omega_0 \|_{L^q}
\end{equation*}
by elliptic regularity theory. Thus, we have
\begin{theorem} \label{theorem2}
$\frac{d}{dt} d( X_1 , \tilde{X}_1 ) \le C \| \omega_0 \|_{L^1 \cap \mathbf{Y}_{\Theta} } \mu_{\Theta} (d( X_1 , \tilde{X}_1 ) ) + C_q  \| \nabla \Phi - \mathbb{I} \|_{C^1} C(\| \nabla \Phi \|_{C^1} ) \| \omega_0 \|_{L^q}$.
\end{theorem}

\vspace{2mm}

\noindent \textbf{Acknowledgements.}  We are grateful to Peter Constantin for his guidance and mentorship over the years.  The work of TDD was partially supported by the NSF CAREER award \#2235395, a Stony Brook University Trustee’s award as well as an Alfred P. Sloan Fellowship.

\vspace{-2mm}

\qquad\\
Theodore D. Drivas\\
Department of Mathematics\\
Stony Brook University, Stony Brook NY 11790, USA\\
{\tt tdrivas@math.stonybrook.edu}

\qquad\\
Joonhyun La\\
School of Mathematics\\
Korean Institute for Advanced Study, Korea\\
{\tt joonhyun@kias.re.kr}

\end{document}